\documentclass{amsart}[12pt]
\def\doctype{}

\usepackage{latexsym,amssymb,dsfont}
\usepackage{tikz}
\usepackage{color}
\usepackage{fancyhdr}
\usepackage{hyperref}
\hypersetup{
colorlinks=true,
allcolors=blue}

\newcommand\Q{\mathbb{Q}}

\newcommand\perm[1]{{\tt #1}}

\newcommand{\comment}[1]{}

\numberwithin{equation}{section}

\fancypagestyle{titlepage}{
\fancyhead[R]{\doctype}
\fancyhead[CL]{}
\cfoot{\vspace{5pt} \thepage}
}

\let\oldsection\section
\newcommand\boldsection[1]{\oldsection{\bf #1}}
\newcommand\starsection[1]{\oldsection*{\bf #1}}
\makeatletter
\renewcommand\section{\@ifstar\starsection\boldsection}
\makeatother

\newtheoremstyle{theorem}
  {12pt}		  
  {0pt}  
  {\sl}  
  {\parindent}     
  {\bf}  
  {. }    
  { }    
  {}     
\theoremstyle{theorem}
\newtheorem{theorem}{Theorem}[section]  
\newtheorem{lemma}[theorem]{Lemma}     

\newtheorem{prop}[theorem]{Proposition}

\newtheoremstyle{definition}
  {12pt}		  
  {0pt}  
  {}  
  {\parindent}     
  {\bf}  
  {. }    
  { }    
  {}     
\theoremstyle{definition}

\newtheorem{ex}[theorem]{Example}

\renewcommand{\proofname}{Proof}

\makeatletter
\renewenvironment{proof}[1][\proofname]{\par
  \pushQED{\qed}%
  \normalfont \partopsep=\z@skip \topsep=\z@skip
  \trivlist
  \item[\hskip\labelsep
        \scshape
    #1\@addpunct{.}]\ignorespaces
}{%
  \popQED\endtrivlist\@endpefalse
}
\makeatother

\makeatletter
\renewcommand*\@maketitle{%
  \normalfont\normalsize
  \@adminfootnotes
  \@mkboth{\@nx\shortauthors}{\@nx\shorttitle}%
  \global\topskip42\p@\relax 
  \@settitle
  \ifx\@empty\authors \else {\vskip 1em
\vtop{\centering\shortauthors\@@par}} \fi
  \ifx\@empty\@date \else {\vskip 1em \vtop{\centering\@date\@@par}}\fi 
  \ifx\@empty\@dedicatory
  \else
    \baselineskip18\p@
    \vtop{\centering{\footnotesize\itshape\@dedicatory\@@par}%
      \global\dimen@i\prevdepth}\prevdepth\dimen@i
  \fi
  \@setabstract
  \normalsize
  \if@titlepage
    \newpage
  \else
    \dimen@34\p@ \advance\dimen@-\baselineskip
    \vskip\dimen@\relax
  \fi
} 
\renewcommand*\@adminfootnotes{%
  \let\@makefnmark\relax  \let\@thefnmark\relax
  \ifx\@empty\@subjclass\else \@footnotetext{\@setsubjclass}\fi
  \ifx\@empty\@keywords\else \@footnotetext{\@setkeywords}\fi
  \ifx\@empty\thankses\else \@footnotetext{%
    \def\par{\let\par\@par}\@setthanks}%
  \fi
\thispagestyle{titlepage}
}
\makeatother

\makeatletter \let\c@table\c@figure \makeatother

\begin{document}

\title[]{\large Fuzzy latin squares and balanced\\ permutation pattern statistics}

\author{Joy Cooper and Peter J.~Dukes}
\address{
Mathematics and Statistics,
University of Victoria, Victoria, BC, Canada
}
\email{joycooper@uvic.ca; dukes@uvic.ca}

\date{\today}

\begin{abstract}
A latin square of order $n$ can be viewed as a partition of the $n \times n$ all-ones matrix into permutation matrix summands.  Here, we consider a relaxation in which the matrix summands are allowed to be induced from shorter permutations.  For $\sigma \in S_k$, the `fuzzy permutation matrix' $P_\sigma^{\uparrow n}$ arises from combining all $\binom{n}{k}^2$ order-preserving embeddings of the $k \times k$ permutation matrix $P_\sigma$ into an $n \times n$ matrix.  We define a fuzzy latin square as a linear combination of $n \times n$ fuzzy permutation matrices $P_\sigma^{\uparrow n}$ equaling a constant matrix.
We study various aspects of these objects, including certain relevant vector space dimensions and a census of fuzzy latin squares with a small number of terms.  In particular, we determine strong conditions on four-term fuzzy latin squares in the `vanishing' case (when the constant matrix is all zeros). We also report on a computer-assisted classification of six-term fuzzy latin squares in the non-vanishing case.
\end{abstract}

\thanks{Research of Peter Dukes is supported by NSERC Discovery Grant RGPIN-2024-03966.}

\maketitle
\hrule

\bigskip

\section{Introduction}
\label{sec:intro}

A \emph{latin square} of order $n$ is an $n \times n$ array with entries from a set of $n$ symbols (often taken to be 
$[n]:=\{1,2,\dots,n\}$) having
the property that each symbol appears exactly once in every row and every column.  

A latin square of order $n$ is equivalent to an ordered decomposition of the $n \times n$ all-ones matrix $J=J_{n \times n}$ into permutation matrix summands.  Since the ordering of terms in this decomposition is unimportant for our purposes, we focus on 
\emph{row-normalized} latin squares, in which the first row is in the natural order ${\tt 123\cdots n}$.  A \emph{normalized} (or \emph{reduced}) latin square has both its first row and first column in the natural order.

As usual, $S_n$ denotes the symmetric group of all permutations on $[n]$.  For $\sigma \in S_n$, the $n \times n$ permutation matrix $P_\sigma$ has entries
$$P_\sigma(i,j)=\begin{cases} 1 & \text{if } \sigma(i)=j, \\ 
0 & \text{otherwise}.
\end{cases}$$
Given a latin square $L$, the symbol `level set' permutations
$\{\sigma_1,\dots,\sigma_n\}$ satisfy
\begin{equation}
\label{eq:ls-cover}
P_{\sigma_1}+\dots+P_{\sigma_n} = J.
\end{equation}
By abuse of notation, we identify a row-normalized latin square with its set of symbol permutations.  In other words, we say $\sigma \in L$ if and only if $P_\sigma$ appears as a term in \eqref{eq:ls-cover}.  An alternate convention could identify $L$ with its set of columns; \eqref{eq:ls-cover} holds in either case.

We now consider a generalization of \eqref{eq:ls-cover} motivated by \cite{CDN2024} and prior work on quasirandom permutations.  
For $\sigma \in S_k$, $1 \le k \le n$, the $n \times n$ \emph{fuzzy permutation matrix} for $\sigma$ is defined by
\begin{equation}
\label{eq:fuzzy-defn}
P^{\uparrow n}_\sigma (x,y) = \frac{(n-k)!}{\binom{n}{k}}\sum_{i=1}^{k} \binom{x-1}{i-1} \binom{n-x}{k-i} \binom{y-1}{\sigma(i)-1} \binom{n-y}{k-\sigma(i)}.
\end{equation}
To motivate the definition, $P^{\uparrow n}_\sigma$ is, up to scalar multiple, the sum of all $n \times n$ partial permutation matrices having exactly $k$ ones in the same relative order as $\sigma$. Fixing an entry $(x,y)$, the product of binomial coefficients in \eqref{eq:fuzzy-defn} counts the ways to place the remaining $k-1$ ones in a partial permutation matrix while respecting the ordering of $\sigma$.

We note that when $k=n$, the sum in \eqref{eq:fuzzy-defn} has its $i$th term equaling $1$ if $(x,y)=(i,\sigma(i))$, and otherwise the term vanishes.  So $P^{\uparrow n}_\sigma=P_\sigma$ in this case.  The multiple in front of the sum equals $1$ for $k=n$ and in general it is a normalization factor calibrated so that the row and column sums of $P^{\uparrow n}_\sigma$ equal $(n-1)!/(k-1)!$. This exact constant is not important here, but it is used in \cite{CDN2024} and is a reasonable choice for the quasirandomness application.

\begin{ex}
Let $\sigma=\perm{213}$ and suppose $n=4$.  There are $\binom{4}{3}^2=16$ partial permutations matching the `medium-low-high' ordering of $\sigma$,
namely 
\perm{213-}, \perm{214-}, \perm{314-}, \perm{324-},
\perm{21-3}, \perm{21-4}, \perm{31-4}, \perm{32-4},
\perm{2-13}, \perm{2-14}, \perm{3-14}, \perm{3-24},
\perm{-213}, \perm{-214}, \perm{-314}, \perm{-324}.
The sum of corresponding partial permutation matrices is
$$\begin{bmatrix}
0 & 1 & 0 & 0\\
1 & 0 & 0 & 0\\
0 & 0 & 1 & 0\\
0 & 0 & 0 & 0\\
\end{bmatrix}
+
\begin{bmatrix}
0 & 1 & 0 & 0\\
1 & 0 & 0 & 0\\
0 & 0 & 0 & 1\\
0 & 0 & 0 & 0\\
\end{bmatrix}
+\dots+
\begin{bmatrix}
0 & 0 & 0 & 0\\
0 & 0 & 1 & 0\\
0 & 1 & 0 & 0\\
0 & 0 & 0 & 1\\
\end{bmatrix}
=
\begin{bmatrix}
0 & 6 & 6 & 0\\
6 & 4 & 2 & 0\\
6 & 2 & 1 & 3\\
0 & 0 & 3 & 9\\
\end{bmatrix}
=4 P_\sigma^{\uparrow 4}.$$
\end{ex}

Informally, $P_\sigma^{\uparrow n}$ `smears out' the permutation $\sigma \in S_k$ onto an $n \times n$ frame.  Figure~\ref{fig:smear-id} shows the effect of this on the identity permutation,
where the grayscale intensity indicates the relative magnitude of entries.

\begin{figure}[htbp]
\begin{center}
\begin{tikzpicture}
\node at (0,0) {\includegraphics[height=2.7cm]{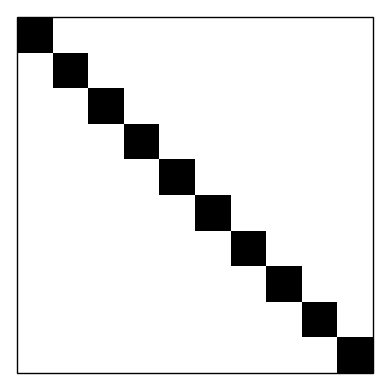}};
\node at (2,0) {\Large $\rightsquigarrow$};
\node at (4,0) {\includegraphics[height=2.7cm]{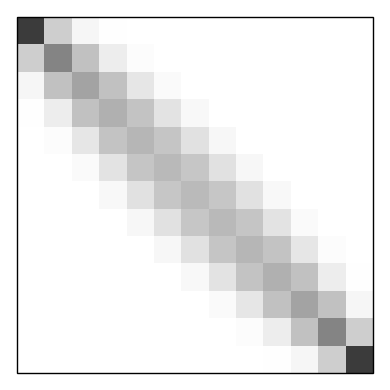}};
\end{tikzpicture}
\caption{The matrices $P_\sigma$ and $P_\sigma^{\uparrow n}$ where $\sigma=\perm{123...k}$.}
\label{fig:smear-id}
\end{center}
\end{figure}

In this paper, we study objects which generalize latin squares in the sense that the permutation matrices in \eqref{eq:ls-cover} are allowed to be fuzzy permutation matrices $P^{\uparrow n}_{\sigma_i}$ for various permutations $\sigma_i$ (of possibly mixed lengths).  In more detail, we let $S_{\le n}:=\cup_{k=1}^n S_k$ and define a 
\emph{fuzzy latin square} of order $n$ as a function $c:S_{\le n} \rightarrow \Q$ such that
\begin{equation}
\label{eq:fuzzy-ls}
\sum_{\sigma \in \mathcal{S}_{\le n}} c(\sigma) P_\sigma^{\uparrow n} = cJ,
\end{equation}
a constant matrix, for some $c \in \Q$.

\begin{ex}
The combination $\perm{3412}+\perm{2143}+\perm{1234}+\perm{4321}$ represents a latin square of order $4$.  If we scale by $3$ and replace the last two terms with $4(\perm{123})+3(\perm{21})$, then the corresponding linear combination of fuzzy permutation matrices equals $9J$. The constituent matrices for each square are shown in Figures~\ref{fig:latin} and \ref{fig:fuzzy}. In the latter case, the shading intensity captures relative entry size, but the overall magnitude of entries is larger.
\end{ex}
\begin{figure}[htbp]
\begin{center}
\begin{tikzpicture}[scale=0.5]
\definecolor{myred}{RGB}{220,60,60}
\definecolor{mygreen}{RGB}{60,180,75}
\definecolor{myblue}{RGB}{70,130,255}
\definecolor{myyellow}{RGB}{255,220,50}

\fill[myyellow] (0,1) rectangle++ (1,1);
\fill[myyellow] (1,0) rectangle++ (1,1);
\fill[myyellow] (2,3) rectangle++ (1,1);
\fill[myyellow] (3,2) rectangle++ (1,1);
\fill[mygreen] (5,2) rectangle++ (1,1);
\fill[mygreen] (6,3) rectangle++ (1,1);
\fill[mygreen] (7,0) rectangle++ (1,1);
\fill[mygreen] (8,1) rectangle++ (1,1);
\fill[myred] (13,0) rectangle++ (1,1);
\fill[myred] (12,1) rectangle++ (1,1);
\fill[myred] (11,2) rectangle++ (1,1);
\fill[myred] (10,3) rectangle++ (1,1);
\fill[myblue] (15,0) rectangle++ (1,1);
\fill[myblue] (16,1) rectangle++ (1,1);
\fill[myblue] (17,2) rectangle++ (1,1);
\fill[myblue] (18,3) rectangle++ (1,1);

\draw[thick] (0,0) grid (4,4);
\draw[thick] (5,0) grid (9,4);
\draw[thick] (10,0) grid (14,4);
\draw[thick] (15,0) grid (19,4);
\end{tikzpicture}
\caption{Matrices $P_\sigma$ in the latin square $\perm{3412}+\perm{2143}+\perm{1234}+\perm{4321}$.}
\label{fig:latin}
\end{center}
\end{figure}
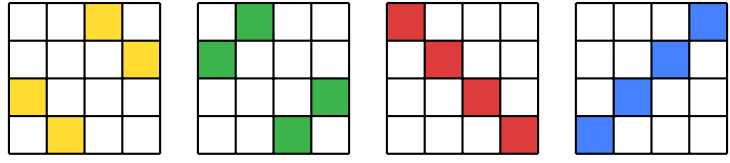

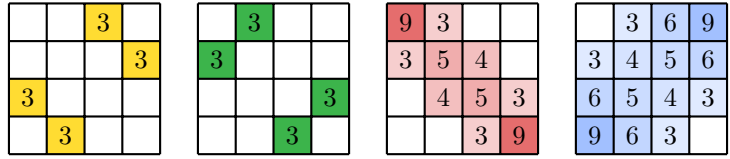
\begin{figure}[htbp]
\begin{center}
\begin{tikzpicture}[scale=0.5]
\definecolor{myred}{RGB}{220,60,60}
\definecolor{mygreen}{RGB}{60,180,75}
\definecolor{myblue}{RGB}{70,130,255}
\definecolor{myyellow}{RGB}{255,220,50}

\fill[myyellow] (0,1) rectangle++ (1,1);
\fill[myyellow] (1,0) rectangle++ (1,1);
\fill[myyellow] (2,3) rectangle++ (1,1);
\fill[myyellow] (3,2) rectangle++ (1,1);
\node at (0.5,1.5) {$3$};
\node at (1.5,0.5) {$3$};
\node at (2.5,3.5) {$3$};
\node at (3.5,2.5) {$3$};

\fill[mygreen] (5,2) rectangle++ (1,1);
\fill[mygreen] (6,3) rectangle++ (1,1);
\fill[mygreen] (7,0) rectangle++ (1,1);
\fill[mygreen] (8,1) rectangle++ (1,1);
\node at (5.5,2.5) {$3$};
\node at (6.5,3.5) {$3$};
\node at (7.5,0.5) {$3$};
\node at (8.5,1.5) {$3$};

\fill[myred!75] (10,3) rectangle++ (1,1);
\fill[myred!25] (11,3) rectangle++ (1,1);
\fill[myred!25] (10,2) rectangle++ (1,1);
\fill[myred!42] (11,2) rectangle++ (1,1);
\fill[myred!33] (12,2) rectangle++ (1,1);
\fill[myred!33] (11,1) rectangle++ (1,1);
\fill[myred!42] (12,1) rectangle++ (1,1);
\fill[myred!25] (13,1) rectangle++ (1,1);
\fill[myred!25] (12,0) rectangle++ (1,1);
\fill[myred!75] (13,0) rectangle++ (1,1);
\node at (10.5,3.5) {$9$};
\node at (11.5,3.5) {$3$};
\node at (10.5,2.5) {$3$};
\node at (11.5,2.5) {$5$};
\node at (12.5,2.5) {$4$};
\node at (11.5,1.5) {$4$};
\node at (12.5,1.5) {$5$};
\node at (13.5,1.5) {$3$};
\node at (12.5,0.5) {$3$};
\node at (13.5,0.5) {$9$};

\fill[myblue!17] (16,3) rectangle++ (1,1);
\fill[myblue!33] (17,3) rectangle++ (1,1);
\fill[myblue!50] (18,3) rectangle++ (1,1);
\fill[myblue!17] (15,2) rectangle++ (1,1);
\fill[myblue!22] (16,2) rectangle++ (1,1);
\fill[myblue!28] (17,2) rectangle++ (1,1);
\fill[myblue!33] (18,2) rectangle++ (1,1);
\fill[myblue!33] (15,1) rectangle++ (1,1);
\fill[myblue!28] (16,1) rectangle++ (1,1);
\fill[myblue!22] (17,1) rectangle++ (1,1);
\fill[myblue!17] (18,1) rectangle++ (1,1);
\fill[myblue!50] (15,0) rectangle++ (1,1);
\fill[myblue!33] (16,0) rectangle++ (1,1);
\fill[myblue!17] (17,0) rectangle++ (1,1);
\node at (16.5,3.5) {$3$};
\node at (17.5,3.5) {$6$};
\node at (18.5,3.5) {$9$};
\node at (15.5,2.5) {$3$};
\node at (16.5,2.5) {$4$};
\node at (17.5,2.5) {$5$};
\node at (18.5,2.5) {$6$};
\node at (15.5,1.5) {$6$};
\node at (16.5,1.5) {$5$};
\node at (17.5,1.5) {$4$};
\node at (18.5,1.5) {$3$};
\node at (15.5,0.5) {$9$};
\node at (16.5,0.5) {$6$};
\node at (17.5,0.5) {$3$};

\draw[thick] (0,0) grid (4,4);
\draw[thick] (5,0) grid (9,4);
\draw[thick] (10,0) grid (14,4);
\draw[thick] (15,0) grid (19,4);
\end{tikzpicture}
\caption{
Matrix terms in the fuzzy latin square $3(\perm{3412})+3(\perm{2143})+4(\perm{123})+3(\perm{21})$.}
\label{fig:fuzzy}
\end{center}
\end{figure}


We emphasize that fuzzy latin squares differ from ordinary latin squares in a few ways.  First, we allow permutations of different lengths; this motivates our adjective `fuzzy'.  Additionally, we allow coefficients in $\Q$, rather than $\{0,1\}$, and the constant $c$ can be arbitrary, instead of $1$.  This fractional relaxation is convenient for facilitating linear algebraic methods.  Finally, latin squares come with an ordering of terms in \eqref{eq:ls-cover}, and we ignore this in our treatment of fuzzy latin squares.  Apart from the key extension to mixed-length permutations, it is possible in principle to impose $c=1$, positive coefficients, and ordered terms in \eqref{eq:fuzzy-ls} to more closely align with the definition of ordinary latin squares. These adjustments would not have significant bearing on our work here.

As mentioned earlier, the definition of $P_\sigma^{\uparrow n}$ was given in \cite{CDN2024} and inspired by permutation pattern statistics.  Suppose $\sigma \in S_k$ and $\pi \in S_n$, where $n \ge k$. A `copy' of $\sigma$ in $\pi$ is a sequence of $k$ distinct indices $1 \le x_1 < x_2 < \cdots < x_k \le n$ on which $\pi(x_i)<\pi(x_j)$ if and only if $\sigma(i)<\sigma(j)$.
The \emph{density} of $\sigma$ in $\pi$, denoted $d(\sigma,\pi)$, is the number of copies of $\sigma$ in $\pi$ divided by $\binom{n}{k}$. We say that a sequence of permutations $\pi_1,\pi_2,\dots$ is \emph{quasirandom} if $\displaystyle \lim_{n \rightarrow \infty} d(\sigma,\pi_n)=1/k!$ for every $\sigma \in S_k$ and every positive integer $k$.

Now consider a linear combination $\alpha=c_1 \sigma_1+\dots+c_r \sigma_r$ of permutations, where $\sigma_i \in S_{k_i}$ for each $i=1,\dots,r$.  If we extend density by linearity, then a quasirandom sequence 
$\{\pi_n\}$ satisfies 
\begin{equation}
\label{eq:quasirandom}
\lim_{n \rightarrow \infty}
d(\alpha,\pi_n)=\sum_{i=1}^r \frac{c_i}{k_i!}.
\end{equation}
In general, $\alpha$ acts as a statistic on permutation patterns, and is called \emph{quasirandom-forcing} if whenever \eqref{eq:quasirandom} holds it is the case that $\{\pi_n\}$ is quasirandom.  It was shown in \cite{CDN2024} (and implicit in prior work on the topic) that $\alpha$ being quasirandom-forcing implies that it is a fuzzy latin square; that is, for some $c$,
\begin{equation}
\label{eq:fuzzy2}    
\sum_{i=1}^r c_i P_{\sigma_i}^{\uparrow n} = cJ.
\end{equation}

In Section~\ref{sec:dimensions}, we study the dimension of vector spaces naturally associated with (fuzzy) latin squares.  This topic is closely connected with deep conjectures, both on irreducible characters of the symmetric group and on the distribution of derangements in latin squares. In Section~\ref{sec:vanishing}, we focus on `vanishing squares', in which the right side of \eqref{eq:fuzzy-ls} is the zero matrix.  We give a near-classification of vanishing squares with four nonzero terms.  In Section~\ref{sec:six-terms}, we classify the solutions to \eqref{eq:fuzzy-ls} with six nonzero terms on the left side, extending prior work in \cite{CDN2024} which had considered five or fewer terms. In our conclusion, we identify some possible next steps for the research.

\section{Dimensions}
\label{sec:dimensions}


\subsection{Set-up}
We work in $\Q[S_n]$ and $\Q[S_{\le n}]$,
the free $\Q$-vector spaces on permutations of length (at most) $n$.  In this section, we study subspaces of these related to our objects of interest.

Define the map $\phi:\Q[S_{\le n}] \rightarrow \Q^{n \times n}$ by
$$\phi \left( \sum_{\sigma \in S_{\le n}} c_\sigma \sigma \right) =
\sum_{\sigma \in S_{\le n}} c_\sigma P_\sigma^{\uparrow n}.$$
That is, $\phi$ replaces permutations 
by their corresponding 
fuzzy permutation matrices, with extension by linearity.  From this, it is easy to see that $\phi$ is a vector space homomorphism.

Let $\mathcal{H}_n \subseteq \Q^{n \times n}$ denote the set of $n \times n$ matrices with constant row and column sums.  It is clear that $\mathcal{H}_n$ is a subspace of $\Q^{n \times n}$. The famous Birkhoff-von Neumann Theorem \cite{Bir1946} states that the set of $n \times n$ doubly-stochastic matrices is the convex hull of the set of $n \times n$ permutation matrices.  Since all $P^{\uparrow n}_\sigma$ have constant line sums, this implies that the image of $\phi$ is precisely $\mathcal{H}_n$.

We have $\mathcal{H}_n = \langle J \rangle \oplus \mathcal{K}_n$, where
$J=J_{n\times n}$ is the all-ones matrix and the latter summand denotes the space of matrices with zero line sums.  In other words, $\mathcal{K}_n$ is the kernel of the map $A \mapsto (AJ,JA)$.  It is easy to see that the image of this map has dimension $2n-1$, so 
$$\dim \mathcal{H}_n = 1+\dim \mathcal{K}_n = 1+n^2-(2n-1) = (n-1)^2+1.$$ 
As an explicit basis for $\mathcal{H}_n$, we can take $J$ together with the $(n-1)^2$ matrices $A^{hk}$, $1 \le h,k < n$, in which 
$$A^{hk}(i,j)=
\begin{cases}
\phantom{-}1 & \text{if } (i,j)=(h,k) \text{ or } (n,n),\\
-1 & \text{if } (i,j)=(h,n) \text{ or } (n,k),\\
\phantom{-}0 & \text{otherwise}.
\end{cases}$$

Let $\mathcal{W}_{\le n}:=\ker \phi$, the set of formal $\Q$-linear combinations of permutations of length $\le n$ whose associated matrix vanishes. Then, 
$\dim \Q[S_{\le n}] = \dim \ker \phi + \dim \mathrm{im}\, \phi
= \dim \mathcal{W}_{\le n} + \dim \mathcal{H}_n$.
It follows that 
\begin{equation}
    \label{eq:dim-Wn}
\dim \mathcal{W}_{\le n} = \sum_{k=1}^n k!-(n-1)^2-1.
\end{equation}

In terms of $\phi$, a fuzzy latin square of order $n$ is an element $f \in \Q[S_{\le n}]$ such that $\phi(f) = cJ$ for some $c \in \Q$. 
Let $\mathcal{F}_{\le n} \subseteq \Q[S_{\le n}]$ denote the space of all fuzzy latin squares with maximum permutation length $n$. Then $\mathcal{F}_{\le n} = \langle \perm{1} \rangle \oplus \mathcal{W}_{\le n}$.  From \eqref{eq:dim-Wn}, we have 
$$\dim \mathcal{F}_{\le n} = \sum_{k=1}^n k!-(n-1)^2.$$
If we define $\mathcal{F}_n = \mathcal{F}_{\le n} \cap \Q[S_n]$, then through similar reasoning as above we get
\begin{equation}
\label{eq:dim-Fn}
\dim \mathcal{F}_{n} = n!-(n-1)^2.
\end{equation}

Recall that each $n \times n$ row-normalized latin square $L$ can be identified with an element of $\Q[S_n]$, namely $\sum_{\sigma \in L} \sigma$.   Let $\mathcal{L}_n$ denote the subspace of $\mathcal{F}_n$ spanned by all $n \times n$ latin squares.  It is natural to ask how tight the containments $\mathcal{L}_n \subseteq \mathcal{F}_n \subseteq \mathcal{F}_{\le n}$ are.  From \eqref{eq:dim-Fn}, we have the bound
\begin{equation}
\label{eq:dim-Ln}
\dim \mathcal{L}_n \le n!-(n-1)^2.
\end{equation}

Our main result of this section is that the upper bound is achieved for sufficiently large $n$.  

\begin{theorem}
\label{thm:dim-Ln}
We have $\dim \mathcal{L}_n = n!-(n-1)^2$ and hence $\mathcal{L}_n=\mathcal{F}_n$ for sufficiently large $n$.
\end{theorem}

The proof of Theorem~\ref{thm:dim-Ln} is given later, in Section~\ref{sec:proof}, after we introduce some background.
Getting a lower bound on $\dim \mathcal{L}_n$ only slightly less than $n!-(n-1)^2$ is relatively straightforward, but closing the gap requires some tools from \cite{CGW2008,KPS2025} on two-row distributions in a random latin square.

\subsection{Latin square completions}

Let $\Lambda_n$ denote the set of all row-normalized latin squares of order $n$ (on symbols $[n]$), and put $\ell_n:=|\Lambda_n|$. The total number of latin squares of order $n$ equals $n! \ell_n$. For permutations $\sigma,\tau \in S_n$, let $$\Lambda_n(\sigma,\tau) =\{L \in \Lambda_n: \sigma,\tau \in L\},$$
and abbreviate $\Lambda_n(\sigma,\sigma)$ by $\Lambda_n(\sigma)$.  We denote the cardinalities by $\ell_n(\sigma):=|\Lambda_n(\sigma)|$ and $\ell_n(\sigma,\tau):=|\Lambda_n(\sigma,\tau)|$.
Since every row-normalized latin square arises from one in $\Lambda_n(\sigma)$ via a unique permutation of rows $2,\dots,n$, it is clear that 
$\ell_n=(n-1)!\ell_n(\sigma)$, independently of $\sigma$.

For $\sigma \neq \tau$, transposing and permuting symbols gives that $\ell_n(\sigma,\tau)$ also counts the latin squares with first row $\sigma$ and second row $\tau$.  
Moreover, under a permutation of symbols and columns, $\ell_n(\sigma,\tau)=\ell_n(\mathrm{id},\delta)$, where $\delta=\sigma^{-1}\tau$, and this depends only on the cycle structure of $\delta$. The count is, of course, zero unless $\delta$ is a derangement.

The quantities $\ell_n(\sigma,\tau)$ are not well understood in general, but
enumerative results for small $n$ appear in \cite{CGW2008}.  The authors of that paper asked whether $\ell_n(\mathrm{id},\delta)$ is, for large $n$, equidistributed across derangements $\delta$.
Peter Cameron's blog also mentions this open problem.

For a set of permutations $\Gamma \subseteq S_n$, let
$$\kappa(\Gamma):=\sum_{\gamma \in \Gamma \setminus \{\mathrm{id}\} } \ell_n(\mathrm{id},\gamma).$$  
For our proof of Theorem~\ref{thm:dim-Ln} to follow, we use asymptotic results on $\kappa$ in some special cases.  One of these concerns $\Gamma=A_n$, the alternating group.  Kwan, Petrova and Sawhney show in  \cite{KPS2025} that two rows in a random latin square of large order $n$ have equal parity with probability tending to $1/2$.  In our language, this says 
$\lim_{n \rightarrow \infty} \kappa(A_n)/\kappa(S_n)=1/2$. With a bit of care, the rate of approach can be estimated using  machinery in that reference.  The following will be proved in Appendix A.

\begin{prop}[from tools in \cite{KPS2025}]
\label{prop:parities}
$$\frac{\kappa(A_n)}{\kappa(S_n)} = \frac{1}{2}+o(1/n).$$
\end{prop}

A partition of $n$ is a list of positive integers $\lambda=(\lambda_1,\dots,\lambda_k)$ with $\lambda_1 \ge \cdots \ge \lambda_k$ and $n=\lambda_1+\cdots +\lambda_k$.  We write $\lambda \vdash n$.  It is convenient to use `exponential' notation to abbreviate partitions; for instance, $1^n$ denotes the partition $(1,\dots,1)$.  The (sorted) list of cycle lengths of a permutation $\sigma \in S_n$ is a partition $\lambda \vdash n$.  We denote the corresponding conjugacy class of $S_n$ by $C_\lambda$. 

A \emph{merge} applied to $\lambda \vdash n$ results in a new partition of $n$ in which some pair of terms $\lambda_i,\lambda_j$ get replaced by their sum $\lambda_i+\lambda_j$ (as a single term).  A \emph{split} is the reverse operation.  The \emph{merge distance} between two partitions $\lambda,\lambda'$ is the minimum number of merge and split operations required to convert $\lambda$ to $\lambda'$.  Cavenagh, Greenhill and Wanless \cite{CGW2008} use split and merge operations to estimate
counts of latin squares with prescribed cycle types.  We later make use of one of their results.

\begin{lemma}[\cite{CGW2008}]
\label{lem:merge}
Suppose $\lambda,\lambda'\vdash n$ have merge distance $d$. Then
$$2^{-d} \le \frac{\kappa(C_\lambda)/|C_\lambda|}{\kappa(C_{\lambda'})/|C_{\lambda'}|} \le 2^d.$$
\end{lemma}

Let $M=M_n$ denote the $n! \times \ell_n$ incidence matrix of $S_n$ versus $\Lambda_n$; that is, $M$ has rows indexed by permutations, columns indexed by latin squares, and entries
$$M(\sigma,L)=\begin{cases} 1 & \text{if } \sigma \in L, \\ 
0 & \text{otherwise}.
\end{cases}$$
The matrix $MM^\top$ has rows and columns
indexed by $S_n$, and the entries are given by
$MM^\top(\sigma,\tau) = \ell_n(\sigma,\tau)$.
Our proof of Theorem~\ref{thm:dim-Ln} is accomplished by showing that most eigenvalues of $MM^\top$ are strictly positive.
For this, we take a brief detour surveying some tools from representation theory of the symmetric group.

\subsection{Characters of $S_n$}

A standard reference for the background material here is Sagan's book \cite{Sagan2001}.  Recall from the previous section that conjugacy classes of $S_n$ are indexed by integer partitions $\lambda \vdash n$.
The irreducible characters of $S_n$ are also indexed by integer partitions $\mu \vdash n$, and we denote the corresponding character by $\chi_\mu$.  If $\sigma \in C_\lambda$, we abuse notation and write $\chi_\mu(\lambda)$ in place of $\chi_\mu(\sigma)$, since characters are constant on conjugacy classes.

For $\lambda \vdash n$, let $G_\lambda$ be the $n! \times n!$ matrix, indexed by $S_n$, with entries
$$G_\lambda(\sigma,\tau)
=\begin{cases}
1 & \text{if } \sigma \tau^{-1} \in C_\lambda;\\
0 & \text{otherwise.}
\end{cases}$$
We survey some well-known facts about these matrices; an early reference is \cite[Chapter 2]{Delsarte1973}.
The $G_\lambda$, $\lambda \vdash n$ are symmetric pairwise commuting matrices.
They have a common basis of eigenvectors, with eigenvalues indexed by the irreducible characters of $S_n$. For each $\mu \vdash n$, $G_\lambda$ has eigenvalue
$|C_\lambda| \frac{\chi_\mu(\lambda)}{\chi_\mu(1^n)}$
with multiplicity $\chi_\mu(1^n)^2$. 
Being a linear combination of the $G_\lambda$, our matrix $MM^\top$ has eigenvalues given by
\begin{equation}
\label{eq:eigenvalues}
\theta_\mu:=
\sum_{\sigma \in S_n} \ell_n(\mathrm{id},\sigma) \frac{\chi_\mu(\sigma)}{\chi_\mu(1^n)} = 
\frac{\ell_n}{(n-1)!}+
\sum_{1 \not \in \lambda \vdash n} \kappa(C_\lambda) \frac{\chi_\mu(\lambda)}{\chi_\mu(1^n)}.
\end{equation}
To clarify notation, the latter sum is over all partitions $\lambda\vdash n$ with parts of size $>1$ (indexing conjugacy classes of derangements in $S_n$).  Recall also that $\kappa(C_\lambda)$ equals the number of row-normalized latin squares whose second row has cycle type $\lambda$.  Observe
\begin{equation}
\label{eq:ls-sum}
\sum_{1 \not\in \lambda \vdash n} \kappa(C_\lambda) = 
\sum_{\sigma \in S_n \setminus \{\mathrm{id}\}} \ell_n(\mathrm{id},\sigma) = 
(n-1) \ell_n(\mathrm{id}) = \frac{\ell_n}{(n-2)!},
\end{equation}
since every reduced latin square is counted $n-1$ times on both sides.  From \eqref{eq:eigenvalues} and \eqref{eq:ls-sum}, a na\"ive lower bound on $\theta_\mu$ is
\begin{equation}
\label{eq:eig-est}
\theta_\mu \ge
\frac{\ell_n}{(n-1)!}+
\frac{\ell_n}{(n-2)!}
\min_\lambda \frac{\chi_\mu(\lambda)}{\chi_\mu(1^n)}
\ge
\frac{\ell_n}{(n-1)!} \left(1-(n-1) 
\frac{\max_\lambda |\chi_\mu(\lambda)|}{\chi_\mu(1^n)} \right).
\end{equation}
Note that the estimate \eqref{eq:eig-est} does not depend on knowledge of $\kappa(C_\lambda)$.

For $\mu=(n-1,1)$, the character $\chi_\mu$ corresponds to the standard representation of $S_n$.  In this case,  $\chi_\mu(\sigma)$ is one less than the number of fixed points of $\sigma$.  For derangements, this evaluates to $-1$, and $\chi_\mu(1^n)=n-1$. Thus, $\theta_\mu=0$ is an eigenvalue of multiplicity $(n-1)^2$.  This provides another proof of \eqref{eq:dim-Ln}.  For general $\mu$, the following estimate on $\chi_\mu$ evaluated at derangements is helpful.

\begin{lemma}[Larsen and Shalev, \cite{LS2008}]
\label{lem:char-est}
Suppose $\lambda \vdash n$ with $1 \not\in \lambda$. Then $|\chi_\mu(\lambda)| \le \chi_\mu(1^n)^{1/2+o(1)}$.
\end{lemma}

\subsection{Proof of Theorem~\ref{thm:dim-Ln}}
\label{sec:proof}

Our goal is to show that the dimension of the $0$-eigenspace of $MM^\top$ equals $(n-1)^2$ for large $n$.  The partition $\mu=(n-1,1)$ corresponds to the standard representation discussed above. Here, $\theta_\mu=0$ with multiplicity $(n-1)^2$.
We argue that, when $n$ is sufficiently large, $\theta_\mu>0$ for all other partitions $\mu \vdash n$, $\mu \neq (n-1,1)$.  

First, from \eqref{eq:eig-est} and Lemma~\ref{lem:char-est}, we get a positive eigenvalue whenever $\chi_\mu(1^n) = \Omega(n^3)$.  By the hook length formula, the only possible exceptions are partitions $\mu \vdash n$ whose Young diagram has `thickness' at most $2$; that is, those having no box coordinates $(i,j)$ with $\min(i,j) > 2$.  Setting aside $(n-1,1)$, the partitions we must consider in more detail are
$$\mu = 1^n, 2^1 1^{n-2}, 2^2 1^{n-4}, (n-2,2), (n).$$

In the simple case $\mu=(n)$, we have $\chi_\mu(\lambda)=1>0$ for all $\lambda \vdash n$ and can immediately conclude $\theta_\mu>0$.  The case $\mu=1^n$ corresponds to the alternating character, with 
$\chi_\mu(\lambda)=\mathrm{sgn}(\lambda)$.  Working from \eqref{eq:eigenvalues} and \eqref{eq:ls-sum},
\begin{align}
\label{eq:alt-eig}
\theta_{1^n} &= \frac{\ell_n}{(n-1)!}+\kappa(A_n)-\kappa(S_n \setminus A_n) 
=\frac{\ell_n}{(n-2)!} \left(\frac{1}{n-1}+\frac{2\kappa(A_n)}{\kappa(S_n)} - 1 \right) \\
\nonumber
&=\frac{\ell_n}{(n-2)!} \left(\frac{1}{n-1} + o(1/n) \right),
\end{align}
where Proposition~\ref{prop:parities} has been used for the last line.  It follows that $\theta_{1^n}>0$ for sufficiently large $n$.

In the remaining cases, formulas for $\chi_\mu(\lambda)$ are easy to compute or estimate from Young tableaux or the Murnaghan-Nakayama rule; \cite{Joy2026} for more details relevant to our cases.
For $\mu=2^1 1^{n-2}$, the hook length formula returns
$\chi_\mu(1^n) = \frac{n!}{n(n-2)!} = n-1$.  
When $n$ is even, $\chi_\mu(n)=1$ and when $n$ is odd $\chi_\mu(3,2^{(n-3)/2})=1$. It follows from \eqref{eq:eig-est} that $\theta_\mu>0$.

Suppose $\mu=(n-2,2)$. We have $\chi_\mu(1^n)=n(n-3)/2$ and $\chi_\mu(2^{n/2})>0$, so again $\theta_\mu>0$.

Suppose $\mu=2^2 1^{n-4}$. As in the previous case, the hook length formula gives $\chi_\mu(1^n)=n(n-3)/2$. For $n \equiv 2 \pmod{4}$, it can happen that $\chi_\mu(\lambda)$ achieves values as small as $-n/2$ when $\lambda=2^{n/2}$. That is, $\min_\lambda \chi_\mu(\lambda)/\chi_\mu(1^n) = -1/(n-3)$. This does not give $\theta_\mu>0$ by \eqref{eq:eig-est} alone, but we can resort to more detailed estimates on characters $\chi_\mu(\lambda)$ and completions $\kappa(C_\lambda)$.  A routine calculation gives the second smallest character ratio.

\begin{lemma}[\cite{Joy2026}]
\label{lem:second-smallest}
Suppose $n \equiv 2 \pmod{4}$, and let $\mu=2^2 1^{n-4} \vdash n$. 
For $\lambda \vdash n, \lambda \neq 2^{n/2}$, the minimum value of $\chi_\mu(\lambda)/\chi_\mu(1^n)$ is $-(n-4)/n(n-3)$, attained at $\lambda=4^1 2^{(n-4)/2}$.
\end{lemma}

For this $\mu$, if we define 
$f(\lambda):=\frac{1}{n-1}+\frac{\chi_\mu(\lambda)}{\chi_\mu(1^n)}$,
then by a minor restatement of Lemma~\ref{lem:second-smallest}, for all $\lambda' \vdash n$, $\lambda' \neq 2^{n/2}$,
\begin{equation}
\label{eq:f-lam}
f(\lambda') \ge \frac{1}{n-1}-\frac{n-4}{n(n-3)} = \frac{2n-4}{n(n-1)(n-3)} = -\frac{n-2}{n} f(2^{n/2})>0.
\end{equation}

To complete the analysis for this eigenvalue, we apply (the right-hand inequality of) Lemma~\ref{lem:merge} with $\lambda=2^{n/2}$ and each of $\lambda'=4^1 2^{(n-4)/2}$, $\lambda''=4^2 2^{(n-8)/2}$,
$\lambda'''=6^1 2^{(n-6)/2}$. Before doing so, we note that for 
$n \ge 8$, we get $|C_{\lambda'}|>n|C_\lambda|$ and similarly for $\lambda''$, $\lambda'''$.
Now, the split-merge inequality gives
$$\kappa(\lambda') \ge n \kappa(\lambda)/2\text{ and }  \ell(1^n,\lambda''),\ell(1^n,\lambda''') \ge n \kappa(\lambda)/4.$$
Using these, \eqref{eq:eigenvalues} and \eqref{eq:f-lam}, we estimate
\begin{align*}
\theta_\mu &>
\sum_{
\lambda^* \in \{\lambda, \lambda',\lambda'',\lambda'''\}}
\kappa(\lambda^*) f(\lambda^*) \\
& \ge
\kappa(\lambda) f(\lambda) -
\kappa(\lambda) \left(\frac{n}{2}+\frac{n}{4}+\frac{n}{4}\right) \frac{n-2}{n} f(\lambda)\\
& = (3-n)  \kappa(2^{n/2}) f(2^{n/2}) = \frac{2}{n-1} \kappa(2^{n/2}) >0.
\end{align*}
In the first line, we have kept only four terms, since the rest are positive.

\subsection{Remarks}

The appendix of \cite{CGW2008} gives counts of $\ell_n(\mathrm{id},\delta)$ for derangements $\delta \in S_n$, $n \le 11$ (according to cycle type of $\delta$). Using these and character tables of $S_n$, we can report that $\dim \mathcal{L}_n = n!-(n-1)^2$ for all $n \le 11$.

Interestingly, we are able to conclude something about the distribution of derangement parities in large latin squares with a purely algebraic argument.  Since $MM^\top$ is positive semidefinite, we know that $\theta_{1^n} \ge 0$.  But from \eqref{eq:alt-eig}, this implies 
$$\frac{\kappa(A_n)}{\kappa(S_n)} \ge \frac{n-2}{2n-2}=\frac{1}{2}-\Theta(1/n).$$
This improves on $\frac{1}{4}-o(1)$ in \cite{CGW2008}, but of course is now subsumed by Proposition~\ref{prop:parities}.

\section{Vanishing squares}
\label{sec:vanishing}

Recall from Section~\ref{sec:dimensions} our definition $\mathcal{W}_{\le n}$ as the set of $\sum_{\sigma} c_\sigma \sigma \in \Q[S_{\le n}]$ such that $\sum_{\sigma} c_\sigma P_\sigma^{\uparrow n}=O$. We call such a $\Q$-linear combination of permutations a \emph{vanishing fuzzy latin square}, or for short a \emph{vanishing square}.  

Vanishing squares are abundant and simple to produce if the number of nonzero terms is unconstrained.  Indeed, for any element $\alpha \in \Q[S_{\le n-1}]$, its associated matrix $\phi(\alpha)$ has constant line sums.  Therefore, from the Birkhoff-von Neumann theorem, there exists $\beta \in \Q[S_n]$ with $\phi(\beta)=-\phi(\alpha)$, and hence $\alpha+\beta \in \ker \phi = \mathcal{W}_{\le n}$.

In the remainder of this section, we collect a few observations on vanishing squares with a prescribed number of terms.

\subsection{Constructions}

We begin with a simple example identified in \cite{CDN2024,KLN2024}.

\begin{ex}
The combination $\perm{1234}-\perm{2134}-\perm{1243}+\perm{2143}$ is a vanishing square with $n=4$, since the associated linear combination of permutation matrices is zero.  More generally, if $X \cup X'$ is a partition of $[n]$ with $|X|,|X'| \ge 2$, $\alpha,\beta$ are distinct permutations of $X$, and $\alpha',\beta'$ are distinct permutations of $X'$, then $\alpha \beta - \alpha' \beta - \alpha \beta' + \alpha'\beta'$ is a vanishing square with four terms.
\end{ex}

In a private communication, the authors of \cite{KLN2024} identified a different class of vanishing squares with four terms.

\begin{ex}
\label{ex:pc}
The combination $\perm{123}-\perm{321}-\frac{3}{4} (\perm{12}-\perm{21})$ is a vanishing square.  The associated combination of $3 \times 3$ matrices is
$$
\begin{bmatrix}
1 & 0 & 0\\
0 & 1 & 0\\
0 & 0 & 1\\
\end{bmatrix}
-\begin{bmatrix}
0 & 0 & 1\\
0 & 1 & 0\\
1 & 0 & 0\\
\end{bmatrix}
-\frac{3}{4} 
\cdot \frac{1}{3}
\begin{bmatrix}
4 & 2 & 0\\
2 & 2 & 2\\
0 & 2 & 4 \\
\end{bmatrix}
+\frac{3}{4}
\cdot \frac{1}{3}
\begin{bmatrix}
0 & 2 & 4\\
2 & 2 & 2\\
4 & 2 & 0\\
\end{bmatrix}
=
\begin{bmatrix}
0 & 0 & 0\\
0 & 0 & 0\\
0 & 0 & 0\\
\end{bmatrix}.
$$
More generally, if $n \ge 3$ is odd, one can produce a vanishing square with permutation lengths $n,n,n-1,n-1$ by joining sub-permutations of length $(n-3)/2$ to the left and right, with those above used on the middle elements.
\end{ex}

The above examples contrast with the situation for non-vanishing fuzzy latin squares. It is not hard to argue that there are only finitely many solutions to \eqref{eq:fuzzy-defn} with $c \neq 0$ and exactly four nonzero terms. Such solutions are classified in \cite[Proposition 4.22]{CDN2024}.

The only three-term vanishing square is $\perm{12}+\perm{21}-2 (\perm{1})$. A simple construction gives vanishing squares with any number of nonzero terms $r \ge 4$.

\begin{prop}
Suppose $r \ge 4$ and $n \ge 2r-4$.  Then there exists a vanishing square with exactly $r$ nonzero terms and all permutation lengths equal to $n$.    
\end{prop}

\begin{proof}
Let $\sigma_1,\sigma_2,\dots,\sigma_{r-2}$ be disjoint transpositions in $S_n$, and put $\sigma:=\sigma_1 \sigma_2 \cdots \sigma_{r-2}$. An element in $[n]$ fixed by $\sigma$ is fixed by all of the $\sigma_i$, while an element not fixed by $\sigma$ appears in exactly one of the transpositions $\sigma_i$.
This implies
$$P_{\sigma_1}+\cdots+P_{\sigma_{r-2}}-P_\sigma=(r-3)I,$$
and hence that $\sigma_1+\dots+\sigma_{r-2}-\sigma-(r-3)\mathrm{id}$ is a vanishing square with $r$ distinct terms.
\end{proof}

\subsection{Four-term vanishing squares}

Here, we give a partial classification of four-term vanishing squares. The broad method is to use bounds on the number of nonzero entries in fuzzy permutation matrices, together with polynomial degree analysis. 

\begin{lemma}[\cite{CDN2024}]
\label{lem:first-row}
Suppose $\sigma \in S_k$ and $n\ge k$.  The first row of $P_\sigma^{\uparrow n}$ contains exactly $n-k+1$ nonzero entries, and any element is repeated at most twice.
\end{lemma}

For a matrix $A$, let $\|A\|_0$ denote the number of nonzero entries of $A$. Although $\| \cdot \|_0$ satisfies the triangle inequality, it fails the homogeneity property of a norm.  Of course, an $n \times n$ permutation matrix $P$ satisfies $\|P\|_0=n$, and the following bounds apply to fuzzy permutation matrices.

\begin{lemma}[\cite{Joy2026}]
\label{lem:zero-bounds}
Suppose $\sigma \in S_k$ and $n\ge k$.  Then
$$n^2-k(k-1) \le \|P_\sigma^{\uparrow n}\|_0 \le k(n-k+1)^2.$$
\end{lemma}

\begin{theorem}
Suppose $c_1 \sigma_1+c_2 \sigma_2+c_3 \sigma_3+c_4 \sigma_4$ is a vanishing square where $\sigma_i \in S_{n_i}$ and wlog $n=n_1 \ge n_2 \ge n_3 \ge n_4 \ge 1$.  Then $(n_1,n_2,n_3,n_4) \in 
\{(n,n,n,n),(n,n,n-1,n-1)\}$ for some integer $n \ge 3$ or
$(n_1,n_2,n_3,n_4) \in \{(3,3,3,2),(3,3,3,1),(4,4,3,2)\}$.
\end{theorem}

\begin{proof}
From \cite{CDN2024}, the only possible permutation lengths satisfy $(n_1,n_2)=(n,n)$ and $$(n_3,n_4) \in \{(n,n),(n,n-1),(n,n-2),(n-1,n-1),(n-1,n-2),(n-1,n-3)\}.$$
It suffices to rule out the cases not covered in the statement.

Suppose $(n_3,n_4)=(n,n-1)$. As a polynomial in either the row index $x$ or column index $y$, $c_1 P_{\sigma_1}^{\uparrow n}+c_2 P_{\sigma_2}^{\uparrow n}+c_3 P_{\sigma_3}^{\uparrow n} = -c_4 P_{\sigma_4}^{\uparrow n}$ has degree $n-2$.  We may therefore assume wlog that $c_1,c_2>0$ and $c_3<0$. From this, we also have $c_4<0$. But then from Lemma~\ref{lem:zero-bounds},
$$3n-2 \le \|c_3 P_{\sigma_3}^{\uparrow n}+c_4 P_{\sigma_4}^{\uparrow n}\|_0 = \|c_1 P_{\sigma_1}^{\uparrow n}+c_2 P_{\sigma_2}^{\uparrow n}\|_0 \le 2n.$$
This forces $n \le 2$, a contradiction.

Suppose $(n_3,n_4)=(n,n-2)$. 
By Lemma~\ref{lem:zero-bounds},
$$5n-6 \le \| P_{\sigma_4}^{\uparrow n} \|_0 = 
\|c_1 P_{\sigma_1}^{\uparrow n}+c_2 P_{\sigma_2}^{\uparrow n}+c_3 P_{\sigma_3}^{\uparrow n}\|_0 \le 3n.$$
This implies $n \le 3$, as required.

Suppose $(n_3,n_4)=(n-1,n-2)$. Wlog, $c_1,c_3>0$ and $c_2,c_4<0$.
Similar to the cases above, we have the bounds
$$5n-6\le\|c_2 P_{\sigma_2}^{\uparrow n}+c_4 P_{\sigma_4}^{\uparrow n} \|_0 = 
\|c_1 P_{\sigma_1}^{\uparrow n}+c_3 P_{\sigma_3}^{\uparrow n}\|_0 \le 4n-2.$$
This implies $n \le 4$.

Suppose $(n_3,n_4)=(n-1,n-3)$. 
By Lemma~\ref{lem:zero-bounds},
$$7n-12 \le \| P_{\sigma_4}^{\uparrow n} \|_0 = 
\|c_1 P_{\sigma_1}^{\uparrow n}+c_2 P_{\sigma_2}^{\uparrow n}+c_3 P_{\sigma_3}^{\uparrow n}\|_0 \le 5n-2.$$
This implies $n \le 5$, and a short computer search rules out any vanishing squares of this type.
\end{proof}

Below we give vanishing squares for the three exceptional tuples. Minor variations on these are also possible; \cite{Joy2026} for more details.

\begin{itemize}
    \item
$(3,3,3,2)$: $\perm{123}+\perm{132}+\perm{213}-\frac{3}{2}\perm{12}$,
    \item
$(3,3,3,1)$: $\perm{123}+\perm{231}+\perm{312}-\frac{3}{2}\perm{1}$,
    \item
$(4,4,3,2)$: $\perm{2143}+\perm{3412}+\frac{4}{3} \perm{123}-\perm{12}$.
\end{itemize}

\section{Fuzzy latin squares with at most six terms}
\label{sec:six-terms}

Consider a fuzzy latin square $\alpha=\sum_{i=1}^r c_i \sigma_i$
in which $c_i \neq 0$ for each $i \in [r]$. We say that $\alpha$ is \emph{irreducible} if, whenever $\alpha'=\sum_{i=1}^r c'_i \sigma_i$ is a fuzzy latin square we have $\alpha'= c \alpha$ for some constant $c \in \Q$.
Up to scaling, the only irreducible fuzzy latin square using the unique permutation in $S_1$ is $\perm{1}$, and the only irreducible fuzzy latin square using both permutations in $S_2$ is $\perm{12}+\perm{21}$.

In this section, we classify the irreducible non-vanishing fuzzy latin squares with six terms.  A similar classification for five or fewer terms was given in \cite{CDN2024}.  From the remarks above, we may ignore permutations of length $1$ and assume that there is at most one occurrence of a permutation of length $2$.  In other words, our classification omits those six-term fuzzy latin squares which arise from simply adding (multiples of) $\perm{1}$ or $\perm{12}+\perm{21}$ to a fuzzy latin square with $r<6$ nonzero terms.  Since our focus is on non-vanishing squares, we assume $c=1$ on the right side of \eqref{eq:fuzzy2}.

The dihedral group $D_4$ acts naturally on $\Q^{n \times n}$ by rotating and reflecting entries.  For instance, reflection in the main diagonal maps $A$ to $A^\top$; rotation by $90^\circ$ clockwise maps $A$ to $B$, with entries $B_{ij}=A_{j,n+1-i}$.  Permutation matrices map to other permutation matrices under any element of $D_4$, so we can consider the induced action on permutations.  As in \cite{CDN2024}, our classification of fuzzy latin squares can be simplified by giving a single representative for each orbit under $D_4$ symmetry.  In particular, we can assume that $\sigma_1(1) \le \lceil n/2 \rceil$.

Let $k_i$ denote the length of $\sigma_i$ for each $i \in [r]$, and assume without loss of generality that $k_1 \ge k_2 \ge \cdots \ge k_r$.  We call $(k_1,\dots,k_r)$ the \emph{length list} of $\alpha = \sum_{i=1}^r c_i \sigma_i$.  If $\alpha$ is a fuzzy latin square in $\Q[S_{\le n}]$, then we may assume $k_1=n$.  In fact, for $n >1$ it is also the case that $k_2=n$, as we discuss below.  

In \cite{CDN2024}, two different constraints were given on the length list of a fuzzy latin square with $r$ terms.  One inequality, for non-vanishing squares, prevents the length list from `staying too large'.  The key idea behind this bound is that
the first row of $J$ on the right side of \eqref{eq:fuzzy2} needs to be covered by nonzero entries from the terms on the left.

\begin{lemma}[\cite{CDN2024}]
\label{lem:list-high}
Suppose $(k_1,\dots,k_r)$ is the length list of a non-vanishing fuzzy latin square.  Then
$\displaystyle n-2 \le \sum_{i=1}^{r-1} (n-k_i+1).$
\end{lemma}

A second constraint prevents the length list from `dropping too quickly' via polynomial degree considerations.  This constraint assumes no partial sum in \eqref{eq:fuzzy2} is vanishing, which is certainly the case for irreducible fuzzy latin squares.

\begin{lemma}[\cite{CDN2024}]
\label{lem:list-drop}
Suppose $(k_1,\dots,k_r)$ is the length list of an irreducible fuzzy latin square.  Then for each $j=1,\dots,r$, we have
$\displaystyle k_j \ge n+1-\sum_{i=1}^{j-1} (n-k_i+1).$
\end{lemma}

For irreducible non-vanishing squares with $r=6$, Lemma~\ref{lem:list-drop} implies $k_1=k_2=n$, $k_3 \ge n-1$, $k_4 \ge n-3$, and $k_5 \ge n-7$. Substituting these into Lemma~\ref{lem:list-high} gives $n \le 18$.  It follows that there are only finitely many candidate length lists.

\begin{table}[htbp]
\small
$$\begin{array}{|l|r|l|r|l|r|}
\hline
\multicolumn{2}{|c|}{n=4} & \multicolumn{2}{c|}{n=5} & \multicolumn{2}{c|}{n=6} \\
\hline
(4, 4, 4, 4, 4, 4) & 33  & (5, 5, 5, 5, 5, 2) & 4   & (6, 6, 6, 6, 5, 4) & 12  \\
(4, 4, 4, 4, 4, 3) & 39  & (5, 5, 5, 5, 5, 4) & 4   & (6, 6, 5, 5, 4, 4) & 43  \\
(4, 4, 4, 4, 4, 2) & 50  & (5, 5, 5, 5, 4, 4) & 10  & (6, 6, 5, 5, 4, 3) & 8   \\
(4, 4, 4, 4, 3, 3) & 36  & (5, 5, 5, 5, 4, 3) & 165 & (6, 6, 5, 5, 4, 2) & 19  \\
(4, 4, 4, 4, 3, 2) & 87  & (5, 5, 5, 5, 4, 2) & 2   & (6, 6, 5, 5, 3, 3) & 8   \\
(4, 4, 4, 3, 3, 3) & 5   & (5, 5, 5, 5, 3, 3) & 2   & (6, 6, 5, 5, 3, 2) & 16  \\
(4, 4, 4, 3, 3, 2) & 12  & (5, 5, 5, 5, 3, 2) & 8   & (6, 6, 5, 4, 4, 4) & 8   \\
(4, 4, 3, 3, 3, 3) & 4   & (5, 5, 5, 4, 4, 4) & 28  & (6, 6, 5, 4, 4, 3) & 8   \\
                   &     & (5, 5, 5, 4, 4, 3) & 22  & (6, 6, 5, 4, 4, 2) & 16  \\
                   &     & (5, 5, 5, 4, 4, 2) & 54  & (6, 6, 5, 4, 3, 3) & 8   \\
                   &     & (5, 5, 5, 4, 3, 3) & 8   & (6, 6, 5, 4, 3, 2) & 32  \\
                   &     & (5, 5, 5, 4, 3, 2) & 26  &                    &     \\
                   &     & (5, 5, 4, 4, 4, 4) & 70  &                    &     \\
                   &     & (5, 5, 4, 4, 4, 3) & 153 &                    &     \\
                   &     & (5, 5, 4, 4, 4, 2) & 162 &                    &     \\
                   &     & (5, 5, 4, 4, 3, 3) & 44  &                    &     \\
                   &     & (5, 5, 4, 4, 3, 2) & 136 &                    &     \\
                   &     & (5, 5, 4, 3, 3, 3) & 2   &                    &     \\
                   &     & (5, 5, 4, 3, 3, 2) & 13  &                    &     \\
\hline
\textbf{Totals:}   & \textbf{266} &  & \textbf{913} &  & \textbf{178} \\
\hline
\end{array}$$
\normalsize
\label{tab:unique6}
\caption{Count of unique solutions, up to $D_4$ equivalence, by length list.}
\end{table}

In our search, we first enumerated $966$ different length lists $(k_1,\dots,k_6)$ with $k_5 \ge 3$ and $k_6 \ge 2$ which are not rejected by either Lemma~\ref{lem:list-high} or \ref{lem:list-drop}.  Each remaining candidate length list is fed into a recursive search which proceeds one row at a time, looping over partial permutations of length $m \ge 1$ in $S_{k_1} \times \cdots \times S_{k_6}$.  Any cases that could not satisfy \eqref{eq:fuzzy2} on the top $m \times n$ submatrix were eliminated before incrementing $m$.  In many cases, especially for larger length lists, an entire tuple could be rejected at $m=1$.  More details on the search methods, including source code, can be found in the first author's thesis \cite{Joy2026} and github repository \cite{github}.

Given permutations $\sigma_1,\sigma_2,\dots$ that admit a solution $c_1,c_2,\dots$ to \eqref{eq:fuzzy2} with $c=1$, there are two possibilities.  Typically, the solution is unique, but occasionally there are infinitely many solutions.  In the latter case, the associated fuzzy latin squares are not irreducible, but we made note of these special solution families with $r=6$.

There are $1128960$ row-normalized latin squares of order six \cite{Frolov1890}.  These correspond to unique solutions for the length list $(6,\dots,6)$.  For other lists with $r=6$, our search identified 1357 unique solutions and several hundred infinite families of solutions producing non-vanishing fuzzy latin squares. Table~\ref{tab:unique6} gives a summary of unique solution counts for each length list.  Table~\ref{tab:positive} gives the irreducible six-term fuzzy latin squares (omitting latin squares) we found having positive coefficients, up to dihedral symmetry and scaling.  A full list of solutions, with permutations and coefficients (scaled as relatively prime integers), is available at \cite{github}. 

\begin{table}[htbp]
\small
$$\begin{array}{|l|}
\hline
5(\perm{1243}) + 4(\perm{1324}) + 5(\perm{2134}) + 4(\perm{2143}) + 6(\perm{3412}) + 4(\perm{321})\\
5(\perm{1234}) + 4(\perm{1324}) + 3(\perm{2143}) + 6(\perm{2413}) + 6(\perm{3142}) + 4(\perm{321})\\
4(\perm{1324}) + 5(\perm{1432}) + 9(\perm{2143}) + 5(\perm{3214}) + 1(\perm{3412}) + 4(\perm{321})\\
2(\perm{1234}) + 4(\perm{1324}) + 3(\perm{1432}) + 6(\perm{2143}) + 3(\perm{3214}) + 3(\perm{21})\\
5(\perm{1243}) + 4(\perm{1324}) + 5(\perm{2134}) + 1(\perm{2143}) + 3(\perm{3412}) + 3(\perm{21})\\
5(\perm{1234}) + 4(\perm{1324}) + 3(\perm{2143}) + 3(\perm{2413}) + 3(\perm{3142}) + 3(\perm{21})\\
6(\perm{1324}) + 5(\perm{1432}) + 8(\perm{2143}) + 5(\perm{3214}) + 2(\perm{4231}) + 3(\perm{21})\\
2(\perm{1234}) + 3(\perm{1243}) + 4(\perm{1324}) + 3(\perm{2134}) + 4(\perm{231}) + 4(\perm{312})\\
12(\perm{1324}) + 15(\perm{1432}) + 24(\perm{2143}) + 15(\perm{3214}) + 8(\perm{123}) + 15(\perm{21})\\
12(\perm{1324}) + 15(\perm{1432}) + 24(\perm{2143}) + 15(\perm{3214}) + 8(\perm{321}) + 3(\perm{21})\\
6(\perm{1243}) + 12(\perm{1432}) + 3(\perm{2143}) + 9(\perm{4321}) + 8(\perm{213}) + 3(\perm{21})\\
5(\perm{1234}) + 4(\perm{1324}) + 3(\perm{2143}) + 4(\perm{132}) + 4(\perm{213}) + 6(\perm{21})\\
3(\perm{1234}) + 3(\perm{4321}) + 2(\perm{132}) + 2(\perm{213}) + 2(\perm{231}) + 2(\perm{312})\\
3(\perm{1324}) + 3(\perm{4231}) + 2(\perm{132}) + 2(\perm{213}) + 2(\perm{231}) + 2(\perm{312})\\
12(\perm{12345}) + 18(\perm{13425}) + 18(\perm{14235}) + 12(\perm{45312}) + 15(\perm{2143}) + 5(\perm{21})\\
18(\perm{13425}) + 18(\perm{14235}) + 12(\perm{15342}) + 12(\perm{42315}) + 15(\perm{2143}) + 5(\perm{21})\\
16(\perm{123456}) + 24(\perm{132546}) + 24(\perm{645231}) + 16(\perm{654321}) + 12(\perm{41352}) + 5(\perm{2413})\\
16(\perm{123456}) + 24(\perm{135246}) + 24(\perm{642531}) + 16(\perm{654321}) + 12(\perm{41352}) + 5(\perm{2413})\\
16(\perm{124356}) + 24(\perm{132546}) + 24(\perm{645231}) + 16(\perm{653421}) + 12(\perm{41352}) + 5(\perm{2413})\\
16(\perm{124356}) + 24(\perm{135246}) + 24(\perm{642531}) + 16(\perm{653421}) + 12(\perm{41352}) + 5(\perm{2413})\\
16(\perm{123456}) + 24(\perm{145236}) + 24(\perm{632541}) + 16(\perm{654321}) + 12(\perm{41352}) + 5(\perm{2413})\\
16(\perm{124356}) + 24(\perm{145236}) + 24(\perm{632541}) + 16(\perm{653421}) + 12(\perm{41352}) + 5(\perm{2413})\\
24(\perm{132546}) + 16(\perm{153426}) + 16(\perm{624351}) + 24(\perm{645231}) + 12(\perm{41352}) + 5(\perm{2413})\\
24(\perm{132546}) + 16(\perm{154326}) + 16(\perm{623451}) + 24(\perm{645231}) + 12(\perm{41352}) + 5(\perm{2413})\\
24(\perm{135246}) + 16(\perm{153426}) + 16(\perm{624351}) + 24(\perm{642531}) + 12(\perm{41352}) + 5(\perm{2413})\\
24(\perm{135246}) + 16(\perm{154326}) + 16(\perm{623451}) + 24(\perm{642531}) + 12(\perm{41352}) + 5(\perm{2413})\\
24(\perm{145236}) + 16(\perm{153426}) + 16(\perm{624351}) + 24(\perm{632541}) + 12(\perm{41352}) + 5(\perm{2413})\\
24(\perm{145236}) + 16(\perm{154326}) + 16(\perm{623451}) + 24(\perm{632541}) + 12(\perm{41352}) + 5(\perm{2413})\\
\hline
\end{array}$$
\normalsize
\label{tab:positive}
\begin{center}
\caption{Irreducible fuzzy latin squares with six terms (of mixed lengths) and positive coefficients.}
\end{center}
\end{table}

\section{Conclusion}
\label{sec:concl}

In this paper, we studied a fractional relaxation of latin squares involving permutations in $S_{\le n}$. These `fuzzy latin squares' form a vector space $\mathcal{F}_{\le n}$ of dimension $1!+2!+\cdots+n!-(n-1)^2$. The set of latin squares of order $n$ span a subspace $\mathcal{L}_n$.

We showed that, for large $n$, $\dim \mathcal{L}_n
=n!-(n-1)^2$.  We also confirmed this formula via computations for $n \le 11$.  This could be extended to larger values with more data on derangement frequencies in latin squares.  

The sequence $n!-(n-1)^2$ appears as OEIS entry \href{https://oeis.org/A372264}{A372264}.  The database identifies these values as the maximum number of distinct cards in a deck that has each card twice and for which the ``$n$ card trick" can be performed.  We could see no obvious connection with latin squares, but it is a noteworthy curiosity.

It may be interesting to determine a natural basis of $1!+2!+\dots+(n-1)!$ fuzzy latin squares for the orthogonal complement of $\mathcal{L}_n$ in $\mathcal{F}_{\le n}$.  In the concrete case $n=5$, a basis of $33$ fuzzy latin squares is reported in the first author's thesis \cite{Joy2026}. Our computations for small $n$ did not seem to reveal any pattern.

Regarding whether $\ell_n(\mathrm{id},\delta)$
is roughly constant over all derangements $\delta \in S_n$, it is possible that the positive semidefiniteness of $MM^\top$ may be useful.  In principle, a linear program could be crafted to study variables $x_\lambda:=\kappa(C_\lambda)$ subject to constraints \eqref{eq:eigenvalues} on the sum of variables and \eqref{eq:ls-sum} on nonnegative eigenvalues.

A row-normalized latin square of order $n$ is equivalent to a one-factorization of the graph $K_{n,n}$.  Each permutation $\sigma \in S_n$ encodes the $1$-factor $\{(i,\sigma(i)):i\in [n]\}$, where edges are listed as ordered pairs. Graphically, a fuzzy permutation matrix $P_\sigma^{\uparrow n}$, where $\sigma \in S_k$, acts as a superposition of all matchings of size $k$ in $K_{n,n}$ with the same order structure as $\sigma$.  If vertices are listed in increasing order, this also relates to the crossing pattern of edges; for instance, 
$P_{\tt 12}^{\uparrow n}$ combines all $2$-edge matchings which do not cross, while $P_{\tt 21}^{\uparrow n}$ does the same for $2$-edge matchings which do cross.  A fuzzy latin square using permutations $\sigma_1,\dots,\sigma_r$ can thus be viewed as a fractional factorization into matching `bundles' which carry the order structure of each $\sigma_i$.  With this viewpoint, it could be interesting to consider a multipartite extension, or in general to consider analogous relaxations of other graph factorization problems.

Another extension of our work could involve a notion of `orthogonality' for fuzzy latin squares.  Two latin squares $L,L'$, each with symbol set $[n]$, are \emph{orthogonal} if 
$\{(L_{ij},L'_{ij}):i,j \in [n]\} = [n]^2$. In terms of permutation matrices, this condition asks for the summands $P_\sigma$ in \eqref{eq:ls-sum} to be as disjoint as possible when comparing terms from $L$ and $L'$.  Stopping short of proposing an explicit definition, we suspect there is a reasonable extension of this for mixed-length permutations. Studying a fuzzy relaxation of mutually orthogonal latin squares could shed light on this challenging topic in design theory.

In \cite{CDN2024}, it was shown that the expression 
\begin{equation}
\label{eq:rho-star}
\rho^*=
\perm{123} + \perm{321} + \perm{2143} + \perm{3412} + \tfrac{1}{2}(\perm{2413}+\perm{3412}) \in \Q[S_{\le 4}]
\end{equation}
is quasirandom-forcing. The same paper shows
that no shorter expression with positive coefficients is quasirandom-forcing.  This was done through a classification and analysis of non-vanishing fuzzy latin squares with at most five terms,  Now, with the classification extended \cite{github} to six terms, a next question would be whether \eqref{eq:rho-star} is the unique six-term quasirandom-forcing expression.  Restricting to expressions with positive coefficients (which includes latin squares of order six) is a reasonable starting point.


\renewcommand{\appendixname}{}
\appendix

\section{Proof of Proposition~\ref{prop:parities}}

Here, we sketch a minor reworking of \cite[Theorem 6.4]{KPS2025} which gives an error term $o(1/n)$ sufficient for our Proposition~\ref{prop:parities}.  This, recall, is used to conclude that the eigenvalue $\theta_{1^n}$ of $MM^\top$ corresponding to the alternating character is positive.  To keep the argument mostly self-contained, we review a few key definitions and lemmas from Sections 7 and 8 of \cite{KPS2025}.

Let $P$ be an $n \times n$ partial latin square.
An \emph{intercalate} in $P$ is a $2 \times 2$ latin subsquare.  One can `switch' the two entries of an intercalate, and the result is a new partial latin square with the same non-empty cells.  Note that an intercalate switch applied to a latin square changes the permutation parity of exactly two rows.

An intercalate in $P$ is \emph{isolated} if it does not share an entry with any other intercalate in $P$.  A set of isolated intercalates is \emph{critical} if, after switching some subset of these, each of the resulting intercalates intersects a common newly created intercalate.  Observe that a critical set has size at most four, because its elements are disjoint and must intersect the same four cells of $P$.  Finally, an isolated intercalate in $P$ is \emph{stable} if it belongs to no critical set.  The following is a restatement of Lemma 7.3 in \cite{KPS2025}.

\begin{lemma}[\cite{KPS2025}]
Suppose $P_1$ is a partial latin square, and $P_2$ is obtained from $P_1$ by switching a stable intercalate $I$. Then the set of stable intercalates in $P_1$ and $P_2$
agree up to swapping symbols in $I$.
\end{lemma}

The next step is to show that a latin square chosen uniformly at random has a sufficient density of stable intercalates.  For $1 \le l \le n$ and $0 <\beta < 1$, an $n \times n$ partial latin square $P$ is an $(l,\beta)$-\emph{intercalate expander} if the following holds: for any sets of rows $R,R'$, any sets of columns $C,C'$, and any sets of symbols $S,S'$ such that five out of six of these have size $\beta n$ and the last has size $l$, there exists a stable intercalate in $P$ with one row in $R$ and the other row in $R'$, one column in $C$ and the other column in $C'$, and one symbol in $S$ and the other symbol in $S'$. An $n \times n$ \emph{template} is a subset $T \subseteq [n]^2$.  For a latin square $L$ of order $n$, $T \cap L$ denotes the partial latin square that agrees with $L$ on the indices of $T$ and whose other cells are blank. We now give a restatement of Lemma 8.3 of \cite{KPS2025}.

\begin{lemma}[\cite{KPS2025}]
\label{lem:temp-exp}
Let $l=\log^{11} n$ and suppose $0<\beta<1$.  There exists a template $T \subset [n]^2$ such that a uniformly random $n \times n$ latin square $L$ satisfies
$$\mathbb{P}[T \cap L \text{ is an }(l,\beta)\text{-intercalate expander}] \ge 1-\exp(-\omega(n \log^2 n)).$$
\end{lemma}

Our last tool is Lemma 8.5 from \cite{KPS2025}, which uses the intercalate expander property to obtain a nearly full set of rows, columns, and symbols on which stable intercalate switches thoroughly mix up the permutation parities.

\begin{lemma}[\cite{KPS2025}]
\label{lem:incidence}
Suppose $P$ is an $(l,1/10)$-intercalate expander, where $l=o(\sqrt{n})$. Then there exist sets $R,C,S$ of rows, columns, and symbols, each of size at least $n-l$, with the following property.  The incidence matrix of $R \cup C \cup S$ versus stable intercalates in $P$ has left kernel over $\mathbb{F}_2$ of dimension $3$, with a basis being $\{(\vec{1},\vec{0},\vec{0}),(\vec{0},\vec{1},\vec{0}),(\vec{0},\vec{0},\vec{1})\}$.
\end{lemma}

The first kernel vector means that stable intercalate swaps within $R$ cannot change the total number of odd rows in $R$.  The other two kernel vectors impose similar constraints on columns and symbols.  However, the kernel having only these vectors means that independently swapping stable intercalates in $P$ produces uniformly random parities between any two rows, provided at least one belongs to $R$.
This is discussed in more detail in Section 8 of \cite{KPS2025}.

We can now closely follow the proof of \cite[Theorem 6.4]{KPS2025} while keeping track of an error bound.
Let $\beta=1/10$.  Take a template $T$ and the value $l$ as given in Lemma~\ref{lem:temp-exp}.  Sample a uniformly random latin square $L$.  The probability that $P=T \cap L$ is not an $(l,\beta)$-intercalate expander is at most $\exp(-\omega(n \log^2 n))=o(1/n)$. Apply Lemma~\ref{lem:incidence} to obtain a set $R$ of at least $n-l$ rows on which switching stable intercalates achieves the row parity mixture discussed above.
The probability that two specific rows (say the first two) lie outside of $R$ is at most $(l/n)^2=o(1/n)$.  
Now, re-randomize $L$ by switching each of its stable intercalates independently at random with probability $1/2$.  The result is a new, still uniformly random, latin square $L'$ in which the first two rows have equal parity with probability $1/2$.

\section*{Acknowledgments}

We are grateful to the authors of \cite{KPS2025} for guidance on sketching the proof of Proposition~\ref{prop:parities}, and to the authors of \cite{KLN2024} for sharing Example~\ref{ex:pc}.


\begin{thebibliography}{99}

\bibitem{Bir1946}
G. Birkhoff,  Tres observaciones sobre el algebra lineal, Univ. Nac. Tucum\'an Rev. Ser. A 5 (1946) 147--151.

\bibitem{CGW2008}
N.~Cavenagh, C.~Greenhill and I.~Wanless, The cycle structure of two rows in a random latin square. \emph{Random Struct. Algorithms} 33 (2008), 286--309.

\bibitem{Joy2026}
J.~Cooper, Fuzzy latin squares as balanced linear combinations of mixed-length permutations,
MSc thesis, University of Victoria, 2026.

\bibitem{github}
J.~Cooper, Repository on fuzzy latin squares,
\url{https://github.com/JoyCooper/FuzzyLatinSquaresWithRTerms}.

\bibitem{CDN2024}
G.~Crudele, P.J.~Dukes and J.A.~Noel,
Six permutation patterns force quasirandomness. 
\emph{Discrete Analysis} 8 (2024), 26 pp.

\bibitem{Delsarte1973}
Ph.~Delsarte, An algebraic approach to the association schemes of coding theory. \emph{Philips Res. Rep.
Suppl.} 10 (1973).

\bibitem{Frolov1890}
M.~Frolov, Sur les permutations carr\'es, \emph{J. de Math. Sp\'ec.} IV (1890), 25--30.

\bibitem{KLN2024}
D.~Kr\'al', J.-B.~Lee and J.A.~Noel, 
Forcing quasirandomness with 4-point permutations, preprint \url{https://arxiv.org/abs/2407.06869}.

\bibitem{KPS2025}
M.~Kwan, K.~Petrova and M.~Sawhney, Parities in random latin squares, preprint \url{https://arxiv.org/abs/2509.13125}.

\bibitem{LS2008}
M.~Larsen and A.~Shalev, Characters of symmetric groups: Sharp bounds and applications. \emph{Inventiones Math.} 174 (2008), 645--687.

\bibitem{Sagan2001}
B.E.~Sagan, The Symmetric Group, 2nd Edition.. Graduate Texts in Mathematics Vol. 203, Springer, 2013.

\end{thebibliography}
\end{document}